\documentclass[12pt,leqno,fleqn]{amsart}  
\usepackage{amsmath,amstext,amsthm,amssymb,amsxtra}
\usepackage{txfonts} 
\usepackage{lmodern}

\usepackage{enumerate}

\usepackage{tikz}

\usepackage{mathtools}
\mathtoolsset{showonlyrefs,showmanualtags}

\usepackage{hyperref} 
\hypersetup{
    colorlinks=true,       
    linkcolor=blue,          
    citecolor=magenta,        
    filecolor=magenta,      
    urlcolor=cyan           
}

\usepackage{amsrefs}

\allowdisplaybreaks
\swapnumbers

\theoremstyle{plain} 
\newtheorem{lemma}[equation]{Lemma} 
 
\newtheorem{theorem}[equation]{Theorem} 
 
\newtheorem{conjecture}[equation]{Conjecture}

\theoremstyle{definition}
\newtheorem{definition}[equation]{Definition} 

\theoremstyle{remark}
\newtheorem{remark}[equation]{Remark}

\numberwithin{equation}{section}

\usepackage{color}

\title{Estimates of the Discrepancy Function in Exponential Orlicz Spaces}
\author{Gagik Amirkhanyan}   

\address{ School of Mathematics, Georgia Institute of Technology, Atlanta GA 30332, USA}
\email {gagik@math.gatech.edu}
\thanks{Research supported in part by  NSF grants DMS-0968499 (G. Amirkhanyan and M. Lacey), DMS-1101519 (D. Bilyk), and a grant from the Simons Foundation \#229596  (M. Lacey).  }

\author{Dmitriy Bilyk }
\address{ School of Mathematics, University of Minnesota, Minneapolis MN 55408, USA}
\email {dbilyk@math.umn.edu}

\author{Michael Lacey}   
\address{ School of Mathematics, Georgia Institute of Technology, Atlanta GA 30332, USA}
\email {lacey@math.gatech.edu}
\begin{document}

\begin{abstract}
We prove that in all  dimensions $ n\ge 3$ for every integer $ N\ge 1$ there exists a distribution of points $ \mathcal P \subset [0,1] ^{n}$ of cardinality $ N$, 
for which the associated discrepancy function $\mathcal  D_N$  
satisfies the estimate 
\begin{equation*}
\lVert \mathcal D_N\rVert_{\operatorname {exp}\big(L ^{\frac 2 {n+1}}\big) } \lesssim (\log N) ^{\frac {n-1}2} \,. 
\end{equation*} 
This has recently been proved by M.~Skriganov, using random digit shifts of binary digital nets, 
building upon the remarkable examples  of  W.L.~Chen and M.~Skriganov.  Our approach, developed independently, 
complements that of Skriganov. 
\end{abstract}

\maketitle

\section{Introduction} 

Given a collection $ \mathcal P$ of $ N$ points in the unit cube in dimension $ n$, the discrepancy function associated to $ \mathcal P$ is defined as 
\begin{equation}\label{e:discrepdef}
\mathcal D_N [\mathcal P, X] := {\sharp}(\mathcal P \cap [0,x]) - N \lvert  [0,X]\rvert \,,  
\end{equation}
where $[0,X]$ is the rectangular box anchored at the origin and $ X= (x_1,...,x_n) \in [0,1] ^{n}$.   
The optimal $ L ^{p}$-estimates for the discrepancy function are well-known, aside from the endpoint cases of $ p=1, \infty  $. 
In this article we continue the  theme begun in \cite{MR2573600} and extend  it to higher dimensions, focusing on the exponential Orlicz space estimates for the 
discrepancy function in dimensions $ n\ge 3$.  

Let $\psi:\, \mathbb R^+ \rightarrow \mathbb R^+$ be an increasing convex function with $\psi (0) = 0$. The Orlicz space $L^\psi$ associated to $\psi$ is  the class of functions for which the norm 
\begin{equation}\label{e:expL0}
\| f \|_{L^\psi}  := \inf \big\{ K>0:\, \int_{[0,1]^n} \psi \big( |f(x)|/K \big) dx \le 1\big\}
\end{equation}
is finite.  The exponential Orlicz spaces $ \operatorname {exp}(L ^{\alpha }) $ are Orlicz spaces associated to the function $\psi$ which equals $e^{x^\alpha}-1$ for large $x$. Exponential Orlicz norms have different equivalent definitions. The one that is most important for this paper is 
\begin{equation} \label{e:expL}
\lVert f\rVert_{\operatorname {exp}(L ^{\alpha })}  \simeq\sup _{q\ge 1} q ^{- \alpha } \lVert f\rVert_{q},
\end{equation}
which allows one to estimate the exponential norm by estimating the $L^q$ norms and carefully keeping track of the constants. 

We prove the following theorem, which as this paper was in final edits, we discovered had been proved by M.~Skriganov \cite{MR2893524}.  

\begin{theorem}\label{t:} In all dimensions $ n\ge 2$ for every integer $ N\ge 1$ there exists a 
distribution $ \mathcal P \subset [0,1]^n$ of   $ N$ points  such that 
\begin{equation} \label{e:}
\lVert D_N\rVert_{\operatorname {exp}\big(L ^{\frac 2 {n+1}}\big)} \lesssim (\log N) ^{\frac {n-1}2} \,. 
\end{equation}
\end{theorem}

It is well-known that the right-hand side of \eqref{e:} is optimal (since  is the best bound for the $L^q$ norms in dimension $n\ge 2$ \cites{MR0066435,MR553291,MR610701}), however,  the left-hand side does not seem to be.   Skriganov \emph{op. cit.}, indicates that 
this conjecture is indeed true: 

\begin{conjecture}\label{j:}
For dimensions $ n\ge 2$, for all integers $ N\ge 1$, there is a choice of $ \mathcal P$ of cardinality $ N$ so that 
\begin{equation} \label{e:}
\lVert \mathcal D_N\rVert_{\operatorname {exp}\big(L ^{\frac 2 {n-1}}\big)} \lesssim (\log N) ^{\frac {n-1}2} \,. 
\end{equation}
\end{conjecture}
In dimension $n=2$ this statement has been proved in \cite{MR2573600} using the digit shifts of the famous van der Corput set.

In dimensions $n\ge 3$, the first \emph{explicit (non-random)}  point distributions  with $\| \mathcal D_N \|_p \lesssim (\log N) ^{\frac {n-1}2}$ are the remarkable examples  obtained by Chen and Skriganov \cites{MR1896098} (in $L^2$) and Skriganov \cite{MR2283797} ($L^p$, $1< p<\infty$), see also {\cite{MR1618631}. In \cites{MR2487691} Chen and Skriganov also considered \emph{random} digit shifts of simpler constructions and showed that they too, on the average, have optimal  $L^2$ norm of  the discrepancy function.
  
The analysis of these constructions exhibits striking similarities to themes related to small ball problems and expansions of the 
Brownian sheet. The fine analysis of these objects is closely related to the (infamous) $ p= \infty $ endpoint analysis 
of the discrepancy function \cite{MR2414745}, also see \cite{MR2817765,1207.6659} for more background on the subject and the techniques.   

A heuristic informed by these connections suggests that the conjecture should be proved 
by estimating the $ L ^{q}$-norm using Littlewood-Paley inequalities $ n-1$ times, with each application giving one square root of $ q$, see \S\ref{s:lp}.     This is just what we will do, but at a specific point in the proof we accumulate one  more power of $ q$. 
In the language of Skriganov, the Littlewood-Paley inequalities are the Khinchin inequalities; in that argument, he applies them $ n$ times. 

The authors discovered the work of Skriganov at the final stages of the editing of this manuscript.  The basic examples are the same nature, but 
there are differences in the details of the proof.  Certainly, the analysis of these examples are subtle, and it may take some time to 
tease out the different variants and details of their analysis.  

In an earlier breakthrough work  \cite{MR2283797} Skriganov showed that for each fixed $ 1< q < \infty $  and integer $ N$ there is 
a \emph{deterministic} distribution $ \mathcal P$ with $ \lVert \mathcal D_N \rVert_{q} \lesssim p^{2n} q ^{ \frac {n+1}2} (\log N) ^{ \frac {n-1}2}$, where $p$ is a prime greater than $qn^2$, hence the real power of $q$ is $\frac{5n+1}2$.%
\footnote{In equation \cite{MR1896098}*{(1.7)}, the estimate is given in terms of a constant in a Littlewood-Paley inequality, which is no more than $ C_n t ^{\frac {n}2}$.} 
In \cite{toAppear}, Skriganov studies the mean behavior of the Discrepancy function, in terms of the shift.  Remarkably, the 
$ L ^{q}$-norms do not depend very much on the choice of the shift.

\section{Linear Distributions} 
Our proofs will assume that $N$ is a power of 2. 
A standard argument then implies the theorem as stated.  
If $2^{s-1} \leq N < 2^s$,  construct a distribution with $2^s$ points in $[0, 1]^n$ with low discrepancy and take $a > 1/2$ such that the cube $[0, a]^n$ contains $N$ points from the distribution.  We get $N$ points in $[0,1]^n$ with low discrepancy by scaling those points inside $[0, a]^n$ by the factor of $1/a < 2$ in each coordinate.  See for instance the beginning of \cite{MR1896098}*{\S3}.  

Let $ U=[0,1]$. 
We shall consider  distributions $ D\subset U ^{n}$ which have the structure of a vector space over the finite field $ \mathbb F _{2}$. 
(More general finite fields can be used, but with this simplest model, the more familiar Rademacher functions reveal themselves.) 
For $ s\in \mathbb N _0$, let $ \mathbb Q (2^s) = \{m 2 ^{-s} \;:\; 0\le m < 2^s\} \subset U$.  Each $ x\in \mathbb Q (2^s)$ can be written in the form 
\begin{equation} \label{e.expand}
x= \sum_{i=1} ^{s} \xi _{i} (x) 2 ^{-s+i-1} = \sum_{i=1} ^{s} \eta _{i} (x) 2 ^{-i}
\end{equation}
with coefficients $ \xi _{i} (x)= \eta _{s-i+1} (x)  \in \mathbb F_2$ for $ 1\le i \le s$. For $ x, y \in \mathbb Q (2^s)$, and $ \alpha ,\beta \in \mathbb F_2$, define $\alpha x \oplus \beta y$ through
\begin{equation*}
\eta_i (\alpha x \oplus \beta y) = \alpha \eta_i( x) + \beta \eta_i (y)  \mod 2 \,, 
\end{equation*}
Then $ \mathbb Q (2^s)$ is a vector space over $ \mathbb F_2$ of dimension $ s$. 

In dimension $ n\ge2$ we consider $ \mathbb Q^{n} (2^s)$ and extend the definition of $ \oplus $ coordinatewise, making  $ \mathbb Q ^{n} (2^s)$ an $ ns$-dimensional vector space over $ \mathbb F_2$.

\begin{definition}\label{d.distribution} We say that $ D \subset  \mathbb Q ^{n} (2^s)$ is a linear distribution if $ D$ is a subspace of $  \mathbb Q^{n} (2^s)$. 
\end{definition}

  The inner product on $  \mathbb Q (2^s) $ 
is defined by 
\begin{equation*}
\langle x,y \rangle=\langle y,x \rangle = \sum_{i=1} ^{s} \xi _{i} (x) \xi _{s-i+1} (y)\,. 
\end{equation*}
This particular structure is dictated by the definition of Walsh functions, see \S \ref{s.walsh}. For $ X = (x_1 ,\dotsc, x_n)$ and $ Y= (y_1 ,\dotsc, y_n) $ in $  \mathbb Q ^{n} (2^s)$, we write 
\begin{equation*}
\langle X,Y \rangle=\langle Y,X \rangle = \sum_{j=1} ^{n} \langle x_j, y_j \rangle \,. 
\end{equation*}
We will frequently write vectors as capital letters  and their coordinates as lower case letters, for example 
$  X = (x_1 ,\dotsc, x_n)$, $ K = (k_1 ,\dotsc, k_n)$, $ L = (\ell _1 ,\dotsc, \ell _n)$, and  this 
convention will be used without further explanation.

For any distribution $ D$ in $  \mathbb Q  ^{n} (2^s)$, we define the \emph{dual distribution} $ D ^{\perp}$ to be the set of  
$ X\in  \mathbb Q ^{n} (2^s)$ with $ \langle X,Y \rangle=0$ for all $ Y\in D$. It follows that $ D ^{\perp}$ is a 
subspace of $  \mathbb Q  ^{n} (2^s)$, hence also a linear distribution. Furthermore, we have $ (D ^{\perp}) ^{\perp}=D$, so that $ D$ and $ D ^{\perp}$ 
are mutually dual distributions. 

Consider the  Rosenbloom-Tsfasman weight defined by 
\begin{equation}\label{e.rt}
\rho (x) = 
\begin{cases}
0, & \textup{if } x=0, \\ \max \{i \;:\; \xi _{i} (x) \neq 0\},  & \textup{if } x\neq 0 ,
\end{cases} 
\end{equation}
i.e. the index of the first non-zero binary digit in the expansion of $ x$. 
It is easy to see that these satisfy the triangle inequality on $  \mathbb Q (2^s) $. They are extended to  $  \mathbb Q ^{n} (2^s)$ by 
the formula $ \rho  (X) = \sum_{i=1} ^{n} \rho  (x_i)$ for $ X= (x_1 ,\dotsc, x_n) \in  \mathbb Q^{n} (2^s)$.
One can check that $ \rho  (X)=0$   iff $ X=0$.

If $ D$ is a linear distribution, we define its Rosenbloom-Tsfasman weight   $ \rho  (D)$ to be the minimum of   $ \rho  (X)$ over  $ X\in D - \{0\}$. 

\begin{remark}\label{r:ham} 
In the works of Chen-Skriganov \cite{MR1896098} and Skriganov \cite{MR2283797}, the more familiar Hamming metric is also used  in order to gain (super) orthogonality relations for integrals of Walsh functions. In our work, as in \cite{MR2893524,MR2487691},  orthogonality is achieved by averaging over random digit shifts instead. 
\end{remark}

\section{Walsh Functions} \label{s.walsh}

We write $ \mathbb Q (2 ^{\infty }) = \bigcup _{s\in \mathbb N _0} \mathbb Q (2^s)$.  The notion of $ \oplus$ addition 
can be defined on this set, making $ \mathbb Q (2^{\infty })$ an infinite dimensional vector space over $ \mathbb F_2$.  
Each $ \lambda \in \mathbb N _0$ can be written as 
$
\sum_{i=1} ^{\infty } \lambda _{i} (\ell ) 2 ^{i-1}
$, 
where the coefficients $ \lambda _i (\ell ) \in \mathbb F_2$ for every $ i \in \mathbb N $ and only finitely many are non-zero.  With this notation, we can 
extend the notion of $ \oplus$ to $ \mathbb N _0$: $ \ell \oplus k $ is the integer $ j$ such that for all $ i\in \mathbb N $, 
\begin{equation*}
\lambda _{i} (j) = \lambda _{i} (\ell ) + \lambda _{i} (k) \mod 2 \,.
\end{equation*}
We define the  Walsh 
functions on $ U$ by 
\begin{equation}\label{e.walsh}
w _{\ell } (x) = \operatorname {exp} \biggl(  { \pi i}  \sum_{i=1} ^{\infty } \lambda _{i} (\ell ) \eta _{i} (x) \biggr)   = (-1)^{\sum_{i=1} ^{\infty } \lambda _{i} (\ell ) \eta _{i} (x)}
\end{equation}
where $ \eta _{i} (x)$ are as in \eqref{e.expand}.  A detailed study of these functions can be found in \cites{MR1117682}.  
The set of functions $ \{w _{\ell } \;:\; \ell \in \mathbb N _0\}$ form an orthonormal basis for $ L ^2 (U)$:  for every $ f\in L ^2 (U)$
\begin{equation*}
f \simeq \sum _{\ell \in \mathbb N _0}  \langle f, w _{\ell } \rangle w _{\ell }
\end{equation*}
with $ \simeq$ indicating that the sum on the right converges to $f$ in the $ L ^2 $ metric.   It is also relevant for us that there is 
an explicit formula connecting Walsh expansions and conditional expectations. 
\begin{equation}\label{e.ce}
\sum_{ \ell =0} ^{2^s-1}  \langle f, w _{\ell } \rangle w _{\ell } 
=  2^s\sum_{t=1} ^{2^s}   \int _{(t-1) 2^{-s} } ^{t 2 ^{-s}} f (y) \; dy \cdot \mathbf 1_{[ (t-1) 2 ^{-s}, t 2 ^{-s})} \,. 
\end{equation} 
It is also the case that $ w _{\ell }$ are the  characters 
of the group $  \mathbb Q (2^{\infty })$.  In particular,  $ w _{\ell } (x \oplus y) = w _{\ell } (x) w _{\ell } (y)$, 
and  $ w _{\ell \oplus k} (x) = w _{\ell } (x) \cdot w _{k} (x) $ for all $ \ell , k \in \mathbb N _0$ and $ x,y\in U$.

\medskip 
In dimension $ n$,  the notion of $ \oplus$ can be extended coordinatewise to $ \mathbb N _0 ^{n}$ 
and likewise to $ \mathbb Q ^{n} (2^{\infty }) $. For $ L = (\ell _1 ,\dotsc, \ell _{n}) \in \mathbb N _0 ^{n}$  
and $ X   \in \mathbb Q^{n} (2^{\infty }) $, we define 
\begin{equation*}
W _{L} (X) = \prod _{j=1} ^{n} w _{\ell _j} (x_j) \,. 
\end{equation*}
The properties mentioned above continue to hold for these Walsh functions.  The collection $ \{W _{L} \;:\; L\in \mathbb N _0\}$ 
forms an orthonormal basis for $ L ^2 (U^n)$, and the $ W_L$ are group characters with respect to $ \oplus$. In particular, 
for all $ L,K\in \mathbb N _0 ^{n}$
\begin{equation*}
\langle W _{L} , W _{K} \rangle = \int _{U ^{n}} W _L   {W _{K}} \; dx 
= \int _{U^n} W _{L\ominus K} \; dx = 
\begin{cases}
1, & L=K, \\ 0, & L \neq K .
\end{cases}
\end{equation*}

There are some useful consequences of $ W _{L}$ being  the group characters, which we collect here.  Consider the vector space over 
$ \mathbb F_2$ given by 
\begin{equation*}
\mathbb N _0 ^{n}( 2^s)  := 
\{ L = (\ell _1 ,\dotsc, \ell _n) \in \mathbb N _0 ^{n} \;:\; 0 \le \ell _i  < 2^s,\ 1\le i \le n \} \,. 
\end{equation*}
Obviously, the map 
\begin{equation*}
\theta \;:\;  \mathbb Q ^{n} (2^s)  \to  \mathbb N_0 ^{n} (2^s)  \;:\; 
(x_1 ,\dotsc, x_n) \mapsto (2^s x_1 ,\dotsc, 2^s x_n)
\end{equation*}
is a vector space isomorphism. The following variant of the Poisson summation formula holds.

\begin{lemma}\label{l.PS} For every linear distribution $ D\subset \mathbb Q^{n} (2^s) $ and every $ L\in \mathbb N _0^{n} (2^s) $, 
it holds that 
\begin{equation*}
\sum_{X\in D} W _{L} (X) = 
\begin{cases}
\sharp D,  &  L \in \theta (D ^{\perp}),   \\  0, & L \not\in \theta (D ^{\perp}).
\end{cases}
\end{equation*}

And for every $X \in \mathbb Q^{n} (2^s) $
\begin{equation*}
\sum_{L\in \theta(D)} W _{L} (X) = 
\begin{cases}
\sharp D,  &  X \in D ^{\perp},   \\  0, & X \not\in D ^{\perp}.
\end{cases}
\end{equation*}


\end{lemma}

Using the isomorphism $\theta$, we  can define $ \rho ( \ell )$ and $ \rho (L)$.  In particular, for $ \ell \in \mathbb N _0 (2^s)$, we write 
$ \ell = \sum_{i=1} ^{s} \lambda _{i} (\ell ) p ^{i-1}$, and then $ \rho (\ell ) $ is the largest $ i$ with $ \lambda_{i} (\ell ) \neq 0$. 
We furthermore set 
\begin{equation}\label{e.zl-zt}
\lambda (\ell ) := \lambda _{\rho (\ell )} (\ell )\,, \qquad \tau  (\ell ):= \ell - \lambda (\ell ) p ^{\rho (\ell )-1} \,. 
\end{equation}
So, $ \lambda (\ell )$ is the most significant digit of $ \ell $, and $ \tau (\ell )$ is $ \ell $ less its most significant term in the 
dyadic expansion of $ \ell $ (we shall say that $\tau (\ell)$ is the \emph{truncation} of $\ell$). For $L\in \mathbb N _0^n (2^s)   $ we set
\begin{equation}
\rho (L) = \sum_{i=1}^n \rho (\ell_i),\qquad \overline{\rho} (L) = \big(\rho(\ell_1),...,\rho(\ell_n)\big),\qquad \tau (L) = \big(\tau(\ell_1),...,\tau(\ell_n)\big).
\end{equation}

\section{Approximation of the Discrepancy Function} 

Let $\chi (y,\cdot)$ be  the indicator of the interval $ [0, y) \subset U$, i.e.
\begin{equation*}
\chi (y,x):= 
\begin{cases}
1  & 0\le x < y  \\ 0 & y\le x < 1 
\end{cases} \,. 
\end{equation*}
This function has Walsh expansion which we write as 
\begin{equation*}
\chi (y,x) \simeq \sum_{\ell \in \mathbb N _0} \widetilde \chi _{\ell } (y)  { w _{\ell } (x) }
\end{equation*}
where $ \widetilde \chi (y) = \langle \chi (y, \cdot ),  {w _{\ell }} \rangle = \displaystyle{\int_0^y \chi(y,x) w_\ell (x) dx}$, and in particular, $ \widetilde \chi _{0} (y)=y$. 
For $ s \in \mathbb N _0$, we truncate the Walsh expansion above to 
\begin{equation*}
\chi _{s} (y,x) = \sum_{ \ell \in \mathbb N _0 (2^s)}  \widetilde \chi _{\ell } (y)   {w _{\ell } (x)}  \,. 
\end{equation*}

This  is extended to $ n$ dimensions.  For $ X, Y\in U ^{n}$, we write 
\begin{align*}
\chi (Y,X) &:= \prod _{j=1} ^{n} \chi (y_j, x_j)\,,
\\
\chi_s(Y,X) &:= \prod _{j=1} ^{n} \chi _s (y_j, x_j)\,,
\\
\mathcal M[D;Y] & := \sum_{X\in D} \chi_s(Y,X) - 2^s \prod _{j=1} ^{n} y_j \,. 
\end{align*}
The first is the indicator of the box in $ U ^{n}$, anchored at the origin and  $ Y$; the second is a truncation of the 
Walsh expansion of the first; and the third is an approximation of the discrepancy function $ \mathcal \mathcal D_N[D,Y]$, since according to \eqref{e:discrepdef}
\begin{align*}
\mathcal D_N [D;Y] & := \sum_{X\in D} \chi (Y,X) - 2^s \prod _{j=1} ^{n} y_j \,.
\end{align*}

For $T \in \mathbb Q ^{n} (2^s)$ the digit shift $D \oplus T$ is defined as  $D \oplus T = \{ X \oplus T:\ X \in D \}$. The following  important observation of Chen and Skriganov \cite{MR1896098}*{Lemma 6A} shows that $\mathcal M[D\oplus T;Y]$ is indeed a good approximation to the discrepancy function.

\begin{lemma}\label{l.approx} Suppose that $ D\subset  \mathbb Q^{n} (2^s) $ is a linear distribution of $N=  2^s$ points 
with dual linear distribution $ D ^{\perp} $ satisfying the bound $ \rho (D ^{\perp}) \ge s - \delta + 1$. We then have 
\begin{equation*}
\lVert \mathcal D_N [D \oplus T,Y] -  \mathcal M[D\oplus T,Y]  \rVert_{ L ^\infty (X) } \le n 2^{\delta}  \lesssim 1 \,. 
\end{equation*}
\end{lemma}

Below, constants that only depend upon the dimension $ n$ will not be systematically tracked.  
The usefulness of this approximation is that $ \mathcal M [D \oplus T;Y]$ can be expressed by a remarkably succinct  formula. Using Poisson summation, Lemma~\ref{l.PS},  we obtain 
\begin{align}
 \mathcal M [D\oplus T;Y] 
 &= \sum_{X\in D} \sum_{L\in \mathbb N _0 (2^s) ^{n}} 
 \widetilde \chi _{L} (Y) {W _{L} (X \oplus T)} - 2^s \prod _{j=1} ^{n} y_j
 \\
 &= \sum_{L\in \mathbb N _0 (2^s) ^{n}}  \Biggl\{
  \sum_{X\in D}{W _{L} (X \oplus T)} \Biggr\} 
 \widetilde \chi _{L} (Y)  - 2^s \prod _{j=1} ^{n} y_j
 \\ \label{e.poisson}
&=   2^s\sum_{L\in \theta (D ^{\perp}) \setminus \{0\}}   W _{L} (T)\widetilde \chi _{L} (Y).
 \end{align}
since $ W _{L} (X \oplus T)= W _{L} (X) W _{L} (T)$ and $ \chi _{\overline{0}} (Y) =  \prod _{j=1} ^{n} y_j$.

Recall $ \widetilde \chi _{L} (Y) = \prod _{j=1} ^{n} \widetilde \chi_{\ell _j} (y_j)$.  
Formulas of Fine \cite{MR0032833} (later extended by Price \cite{MR0087792} to $p$-adic Walsh functions and known as Fine-Price formulas) give a precise expansion of the $ \tilde \chi _{\ell }$. 
For every $ \ell \in \mathbb N _0$,  we have
\begin{align}
\label{e.fine}
 \widetilde \chi _{\ell } (y)  = 2^{- \rho (\ell)} u _{\ell} (y), \,\, \textup{ where  } \,\, u _{\ell } (y) &= 
\tfrac 12  \bigl( w _{\tau ( \ell )} (y) - 
\sum _{i=1} ^{ \infty } 2^{-i}  w _{ \ell + 2^{\rho(\ell)+i-1}} (y) \bigr)  \,. 
\end{align}
The equality above holds for $ \ell =0$ as well, with the understanding that $ \tau (0)= \rho (0)=0$. 
 

Recall that for $  x \in U := [0,1]$, we write 
$ x= \sum_{i=1} ^{\infty } \eta  _{i} 2 ^{-i}$, where $ \eta _{i} (x) \in \{0,1\}$. 
The  Rademacher functions are defined as \begin{equation}\label{e:raddef}
 r _{i} (x) = (-1) ^{\eta _i (x)}.
 \end{equation}
 In particular, $w_{2^k } = r_{k+1}$. We then have the following representation.

\begin{lemma}\label{l.omegaRep}  For any $\ell \in N_0$ we have
\begin{align}
 \widetilde \chi _{\ell } (y) & = 2 ^{- \rho (\ell) - 1} w _{\tau ( \ell )} (y) \omega_{\rho(\ell)}(y),\,\,\, \textup{\emph{where}}
\\ 
 \quad \omega_{\rho(\ell)}(y) & = 1  - 
\sum _{i=1} ^{ \infty } 2 ^{-i} r_{\rho(\ell)}(y) r_{\rho(\ell)+i}(y),
\end{align}
and $r_k(y)$ are the Rademacher functions.
\end{lemma}

The function $ \omega_{\rho(\ell)}(y)$ is continuous and piecewise linear with a period of $2^{-\rho(\ell) + 1}$.

\begin{proof}
As $\ell + 2 ^{\rho(\ell)+i-1} = \tau ( \ell ) \oplus  2 ^{\rho(\ell) - 1}  \oplus 2 ^{\rho(\ell)+i-1}$ then
\begin{equation}
 w _{\ell + 2 ^{\rho(\ell)+i-1}} (y) = w_{\tau ( \ell )}(y) w_{2 ^{\rho(\ell) - 1}}(y)  w_{2 ^{\rho(\ell)+i-1}}(y) = w_{\tau ( \ell )}(y)r_{\rho(\ell)}(y) r_{\rho(\ell)+i}(y).
\end{equation}
Which along with \eqref{e.fine} proves Lemma \ref{l.omegaRep}.
\end{proof}

\begin{remark} Lemma \ref{l.omegaRep} may be explained and proved without appealing to the Fine-Price formula  \eqref{e.fine}. Indeed, the integral of a Rademacher function $\displaystyle{\int_0^y r_k (x) dx}$ is the $2^{-k+1}$-periodic ``saw-tooth" function. Hence, the integral of the Walsh function $w_\ell = r_{\rho(\ell)} \cdot w_{\tau (\ell) }$ also has this structure, but with sign changes on dyadic intervals of length $2^{-k+1}$ dictated by the sign of $w_{\tau (\ell) }$. One can easily check that on $[0,1]$ $x= \frac12 \big( 1- \sum_{i=1}^\infty 2^{-i} r_i (x) \big) $ and therefore the $1$-periodic ``saw-tooth" function $|||x|||$, i.e. the distance from $x$ to the nearest integer, satisfies 
 \begin{equation}
|||x||| =\frac12 \big( 1- \sum_{i=1}^\infty 2^{-i} r_1 (x) r_i (x)  \big) = \frac1{2^2} \big( 1- \sum_{i=1}^\infty 2^{-i} r_1 (x) r_{1+i} (x)  \big). 
\end{equation}
The rest follows by rescaling.
\end{remark}

 \section{The Rademacher Functions and Shifts} 

We say that a distribution $D$  with $N=2^s$ points is a dyadic net with deficiency $\delta$ if each dyadic box of volume $2^{-s+\delta}$ in $U^n$  contains precisely $2^\delta$ points of $D$. It is well known that  is equivalent to the fact that   $ \rho (D ^{\perp}) \ge s - \delta + 1$ (see e.g. Lemma 2C in \cite{MR1896098}). While dyadic nets with deficiency zero do not exist in dimensions $n>3$, one can construct dyadic nets with deficiency $\delta$ of the order $n\log n$ in any dimension. See the book \cite{MR2683394} for a detailed treatment of digital nets.

Assume that $D$ is a dyadic net with deficiency $\delta$ and return to formula \eqref{e.poisson}:
\begin{equation}\label{e.mShifts}
\mathcal M [D \oplus T; Y]  =  2 ^{s}\sum_{L\in \theta (D ^{\perp}) \setminus \{0\}}  W_L (T) \, \widetilde \chi _{L} (Y) 
\end{equation}  
Switch to the vector notation, setting 
$Y = (y_1, ..., y_n)$, $L = (\ell_1, ..., \ell_n)$, $\overline{\rho}(L) = (\rho(\ell_1), ... , \rho(\ell_n))$, $\omega_{\overline{\rho}(L)} (Y) =  \prod\limits_{i = 1}^n \omega_{\rho(\ell_i)} (y_i)$, and $r_{\overline{\rho}(L)} (Y) =  \prod\limits_{i = 1}^n r_{\rho(\ell_i)} (y_i)$.
Applying Lemma~\ref{l.omegaRep}  to the summands above, we obtain 
\begin{align}
W_L (T) \, \widetilde \chi _{L} (Y) &=   2 ^{- n - \rho(L)} W_L (T) W_{\tau(L)} (Y) \omega_{\overline{\rho}(L)} (Y),
\\
&=  2 ^{- n - \rho(L)} r_{\overline{\rho}(L)} (Y) W_L (T) W_L (Y) \omega_{\overline{\rho}(L)} (Y),   
\qquad  (\textup{since }W _{\tau (L)} = r_{\overline{\rho}(L)} W _{L})
\\ 
&=   2 ^{- n - \rho(L)} r_{\overline{\rho}(L)}(Y) \omega_{\overline{\rho}(L)} (Y) W_L (Y \oplus T). 
\end{align} 
Whence  we have 
\begin{equation} \label{e.Mform}
\mathcal M [D \oplus T; Y]  
=  \sum_{L\in \theta (D ^{\perp}) \setminus \{0\}} 2 ^{s - n - \rho(L)} r_{\overline{\rho}(L)}(Y) \omega_{\overline{\rho}(L)} (Y) W_L (Y \oplus T),
\end{equation}
This leads to the following consequence for the $ L ^{q}$ norms of this sum.   (In view of Lemma~\ref{l.approx}, it clearly completes the proof 
of our main theorem, Theorem \ref{t:}.  Indeed, this inequality implies that $\mathcal D_N [D\oplus T, Y]$ satisfies  the $\operatorname{exp} \big( L^{\frac2{n+1}}\big)$ bound as a function of two variables $Y$ and $T$. Therefore, for some $T$ it has to satisfy this bound in $Y$.)

\begin{lemma} \label{l.MEstimate}
Let the distribution $D$ with $N=2^s$ points be   a dyadic net with deficiency $\delta$. For any $1\leq q < \infty$ we have
\begin{equation}
\Vert \mathcal M [D \oplus T; Y] \Vert_{L_{q}[Y\times T]}  \leq C  q^{\frac {n+1}2}  s^{\frac{n-1}{2}}, 
\end{equation}
where the implicit constant depends only on the dimension $n$ and deficiency $\delta$.
\end{lemma}

\begin{proof}
It is  convenient to prove the lemma for $ q$ replaced by $ 2q$, with $ q \in \mathbb N $.  
The following  elementary fact will be used: 
for an integrable function $f: U ^{n} \rightarrow \mathbb{R}$  and fixed $Z \in U^n$ we have
\begin{equation} \label{e:Y+}
\int_{U^n} f(Y) dY = \int_{U^n} f(Y \oplus Z) dY.
\end{equation}
According to it, it suffices to estimate the $  L ^{2q}[Y \times T]$ norm of $  \mathcal M [D \oplus T; Y \oplus T]$. 

The latter has a more symmetric expansion.   From \eqref{e.Mform} we get
\begin{align} \label{e.mShifts1}
\mathcal M [D \oplus T; Y \oplus T] & =  \sum_{L\in \theta (D ^{\perp}) \setminus \{0\}} 2 ^{s - n - \rho(L)} r_{\overline{\rho}(L)} (Y \oplus T) \omega_{\overline{\rho}(L)} (Y \oplus T) W_{L} (Y), 
\\
\intertext{by grouping the summands in \eqref{e.mShifts1} which have the same $\overline{\rho}(L)$ we obtain}
\label{e.mShifts2}
&=  \sum_{\overline{\rho} \in \mathbb N (2^s) ^{n},\ |\overline{\rho}| > s - \delta} 2 ^{s - n - | \overline{\rho} |} r_{\overline{\rho}} (Y \oplus T)  \omega_{\overline{\rho}} (Y \oplus T) \sum_{L\in \Lambda(\overline{\rho})} W_{L} (Y) , 
\end{align}
where $|\overline{\rho}| = \rho_1+...+\rho_n$ is the $\ell^1$ norm of $\overline{\rho}$ and $\Lambda(\overline{\rho}) = \{ L\in \theta (D ^{\perp}): \ \overline{\rho}(L) = \overline{\rho} \}$.
The latter  is an affine copy of the subspace  
\begin{equation}
\Lambda_0(\overline{\rho}) = \{ L\in \theta (D ^{\perp}): \ \overline{\rho}(L) < \overline{\rho} \}\,. 
\end{equation}
The cardinality of $\Lambda_0(\overline{\rho})$ satisfies 
 \begin{equation}
\sharp \Lambda_0(\overline{\rho}) \leq 2^{|\overline{\rho}| - s + \delta}.
\end{equation}
To see this, observe that  $\Lambda_0(\overline{\rho}) $ is 
$ \theta (D ^{\perp})$ restricted to a  dyadic box of area $2^{|\overline{\rho}|}$. 
Divide the box 
into $2^{|\overline{\rho}| - s + \delta}$ congruent boxes of volume $2^{s-\delta}$. Since $ \rho (D ^{\perp}) \ge s - \delta + 1$, each such box  contains no more than one point of $\theta (D^\perp)$ (for otherwise the difference of the two points would yield a non-zero point of $L = L_1 \ominus L_2 \in \theta(D^\perp)$ with  $\rho (L) \le s-\delta$). \\

The sum $\sum\limits_{L\in \Lambda(\overline{\rho}) } W_{L} (Y)$  can be written using the Poisson summation formula Lemma \ref{l.PS}. 
\begin{equation} \label{e.rhoSum}
\sum_{L\in \Lambda(\overline{\rho})} W_{L} (Y) =  W_{L_{\overline{\rho}}} (Y) \sum_{L\in \Lambda_0(\overline{\rho}) } W_{L} (Y) = 
W_{L_{\overline{\rho}}} (Y) \delta(\overline{\rho}, Y),
\end{equation}
 where 
 $L_{\overline{\rho}}$ is any point in $\Lambda(\overline{\rho})$ and
\begin{equation} \label{e.deltaPerp}
 \delta(\overline{\rho}, Y) = 
\begin{cases}
\sharp \Lambda_0(\overline{\rho}),  &  Y \perp  \{X \in D ^{\perp}:   \ \overline{\rho}(X) < \overline{\rho} \},   \\  
0, & \textup{otherwise}.
\end{cases}
\end{equation}
The orthogonality condition in \eqref{e.deltaPerp} makes sense if we consider $Y \in \mathbb Q^{n} (2 ^{s}) $ by truncating the extra  binary digits (above $s^{th}$) in each coordinate. 
  We can easily see that 

\begin{equation}  \label{e.deltaInt}
\int_{[0,1]^n} \delta(\overline{\rho}, Y) dY = \sharp \Lambda_0(\overline{\rho}) \cdot \sharp \big[ \Lambda_0(\overline{\rho})\big]^\perp \cdot 2^{-ns} = 1 .
\end{equation}  

Combining \eqref{e.mShifts2} and \eqref{e.rhoSum} we obtain
\begin{align}
\mathcal M [D \oplus T; Y \oplus T] & =   \sum_{\overline{\rho} \in \mathbb N (2^s) ^{n} ,\ |\overline{\rho}| > s}   2 ^{s - n - | \overline{\rho} |} r_{\overline{\rho}} (Y \oplus T)  \omega_{\overline{\rho}} (Y \oplus T)  W_{L_{\overline{\rho}}} (Y) \delta(\overline{\rho}, Y) 
\\  \label{e:AA}
& =  2^{- n } \sum_{k=s - \delta + 1}^{ns} M_k(T, Y) \,, 
\\
\textup{where} \quad 
M_k (T, Y) & := \sum_{\overline{\rho} \in \mathbb N (2^s) ^{n} ,\ |\overline{\rho}| = k}  2 ^{s  -k} \, r_{\overline{\rho}} (Y \oplus T)  \omega_{\overline{\rho}} (Y \oplus T)  W_{L_{\overline{\rho}}} (Y) \delta(\overline{\rho}, Y).
\end{align}

The variables $ Y$ and $ T$ can be decoupled.   By \eqref{e:Y+},  the  $L_{2q}[Y \times T]$ norm of $ M_k [D \oplus T; Y \oplus T]$  equals the $L_{2q}[Y \times T]$ norm of 
\begin{equation*}
M_k'(T, Y) = \sum_{\overline{\rho} \in \mathbb N (2^s) ^{n} ,\ |\overline{\rho}| = k} 2 ^{s  -k} \,  r_{\overline{\rho}} (T)  \omega_{\overline{\rho}} (T)  W_{L_{\overline{\rho}}} (Y) \delta(\overline{\rho}, Y) \,. 
\end{equation*}
Let us rewrite the function $ r_{\overline{\rho}} (T)  \omega_{\overline{\rho}} (T) $.   Using Lemma \ref{l.omegaRep}, since $r_m^2 = 1$, we have 
\begin{equation}
r_{m}(t) \omega_{m}(t) = r_{m}(t)  - 
\sum _{i=1} ^{ \infty } 2 ^{-i} r_{m+i}(t) .
\end{equation}
This implies that 
\begin{equation}
r_{\overline{\rho}} (T)  \omega_{\overline{\rho}} (T) = \sum_{\overline{\imath} \in \mathbb{N}_0^n} \epsilon_{\overline\imath} 2^{-| \overline{\imath} |} r_{ \overline{\rho} +  \overline{\imath}} (T).
\end{equation}
Here, $\epsilon_{\overline\imath} $ is $ -1$ raised to the number of non-zero entries of $ \overline \imath$.  
Therefore 
\begin{align} \label{e.MkPrime}
M_k'(T, Y) &= \sum_{\overline{\imath} \in \mathbb{N}_0^n} \epsilon_{\overline\imath} 2^{-| \overline{\imath} |}  \sum_{\overline{\rho} \in \mathbb N (2^s) ^{n} ,\ |\overline{\rho}| = k} 2 ^{s  - k} \,  r_{ \overline{\rho} +  \overline{\imath}}  (T) W_{L_{\overline{\rho}}} (Y) \delta(\overline{\rho}, Y) \\
&=: \sum_{\overline{\imath} \in \mathbb{N}_0^n} \epsilon_{\overline\imath} 2^{-| \overline{\imath} |} M_k^{\overline{\imath}}(T, Y). 
\end{align}

We estimate the $ L ^{2q} (Y \times T)$ norm of $ M_k^{\overline{\imath}}(T, Y)$, which the 
Littlewood--Paley inequalities are ideally suited for.   Applying Lemma~\ref{l:hyperbolic} in $T$ and using the fact that $ q\in \mathbb N $, we obtain   
\begin{align} 
\Vert M_k^{\overline{\imath}}(T, Y) \Vert_{L_{2q}[T \times Y]}^{2q} &\leq (Cq)^{q(n-1)}\,  2 ^{(s  -k) 2q} \, \int \left( \sum_{\overline{\rho} \in \mathbb N (2^s) ^{n} ,\ |\overline{\rho}| = k} \delta^2(\overline{\rho}, Y) \right)^q \; dY \\
&= (Cq)^{q(n-1)} \,2 ^{(s  -k) 2q}  \int \sum_{\lvert  \overline{\rho}_1\rvert ,\dotsc, \lvert \overline{\rho}_q\rvert = k} \delta^2 (\overline{\rho}_1, Y) \cdots  \delta^2 (\overline{\rho}_q, Y)\; dY  \\
& \le 
(Cq)^{q(n-1)}\, 2 ^{(s  - k) 2q}    \sum_{\lvert  \overline{\rho}_1\rvert ,\dotsc, \lvert \overline{\rho}_q\rvert = k} 
 \prod_{j=2}^q  \big[  \sharp \Lambda_0(\overline{\rho}_j)\big]^2  \cdot  \int \delta^2 (\overline{\rho}_1, Y)  \; dY 
\label{e.lpM}  \\
& \le 
(Cq)^{q(n-1)}\,  2 ^{(s  -k ) 2q}  \,  s^{q(n-1)} \,   2^{(k  - s + \delta)(2q-1)}    \\
 & \leq    (Cq)^{q(n-1)} \, s^{q(n-1)} \, 2^{  s - k  }.
\end{align}
The constant $ C$ changes from line to line above.  The first line is the Littlewood-Paley inequality; the third one uses the fact that $ \delta (\overline \rho ,Y)$ takes values  $ 0$ and $ \sharp \Lambda_0(\overline{\rho})$; the fourth one uses the facts that the number of $ \overline \rho $ with $ \lvert  \overline \rho \rvert=k $ 
is at most $ C k ^{n-1} < C' s ^{n-1}$, that $ \sharp \Lambda_0(\overline{\rho}) \leq 2^{|\overline{\rho}| - s + \delta}$, and that   $$ \int \delta^2 (\overline{\rho}, Y)  \; dY =  \sharp \Lambda_0(\overline{\rho}) \cdot  \int \delta  (\overline{\rho}, Y)  \; dY =   \sharp \Lambda_0(\overline{\rho}) \leq 2^{|\overline{\rho}|  - s + \delta}$$ in view of   \eqref{e.deltaInt}.    


Because of the geometric decay in  \eqref{e.MkPrime} we get
\begin{equation}  
\Vert M_k'(T, Y) \Vert_{L_{2q}[T \times Y]} \leq C q^{\frac{n-1}{2}}  s^{\frac{n-1}{2}}  2^{\frac{s - k   }{2q}} .
\end{equation}
Note in particular the exponent of $ 2$ above, which will lead to one additional power of $ q$ in our estimate. 
Recall that $\Vert M_k(T, Y) \Vert_{L_{2q}[T \times Y]} = \Vert M_k'(T, Y) \Vert_{L_{2q}[T \times Y]}$, thus from \eqref{e:AA}  we obtain
\begin{align} 
\Vert \mathcal M [D \oplus T; Y \oplus T] \Vert_{L_{2q}[T \times Y]}  &\leq C  q^{\frac{n-1}{2}}  s^{\frac{n-1}{2}}   \sum_{k=s+1}^{ns} 2^{\frac{s - k }{2q}}  \\
\label{e:fullPower}
  &\leq C  q^{\frac{n+1}{2}}  s^{\frac{n-1}{2}}    \leq C  q^{\frac{n+1}{2}}  s^{\frac{n-1}{2}} 
\end{align}
This completes the proof. 
\end{proof}

The reader interested in further improvements in arguments of this type will quickly focus on the fact that this method of proof 
uses the Rademacher structure, but exploits very little information (essentially just \eqref{e.deltaInt}) about the coefficients of the Rademacher functions.
The first  point where one would like to do much better is  estimate \eqref{e.lpM} above: here the integral of the $ q$-fold product of $ \delta (\overline  \rho ,Y) $ 
is estimated by the integral of a single $ \delta (\overline  \rho ,Y)$.   
However we have only found incremental improvements on this point  and we leave the topic to the future.  

At this point, it might be convenient to point out why Conjecture~\ref{j:} represents a natural goal, and why the 
possible extensions are far from clear.  For integers $ k$, one has 
\begin{equation*}
\Bigl\lVert \sum_{\overline  \rho \::\: \lvert  \overline \rho \rvert=k } r _{\overline \rho } \Bigr\rVert_{k} 
\gtrsim k ^{ (n-1) - \frac n k}, 
\end{equation*}
since on the cube $ [0, 2 ^{-k}] ^{n}$, the summands are all of the same sign.  On the other hand, the Littlewood-Paley immediately show that 
$ \Bigl\lVert \sum_{\overline  \rho \::\: \lvert  \overline \rho \rvert=k } r _{\overline \rho } \Bigr\rVert_{ \operatorname {exp}(L ^{\frac 2 {n-1}})} \lesssim k ^{\frac {n-1}2}$, which by the above is not improvable.

\section{The Littlewood--Paley Inequalities} \label{s:lp}

We start with the following  version of the Littlewood--Paley inequalities (which is just the Hilbert space-valued Khinchin inequality):

\begin{lemma}\label{l:lp} For coefficients $ c _{i} $ in a Hilbert space $ \mathcal H$ and for any $q\ge2$, there holds 
\begin{equation*}
\Bigl\lVert \sum_{i} c_i r _{i} \Bigr\rVert_{L _{q} (U)} 
\le C \sqrt q   \Bigl[\sum_{i} \lvert  c_i \rvert ^2   \Bigr] ^{\frac{1}{2}}   \,,
\end{equation*}
where $r_i$ are the Rademacher functions as defined in \eqref{e:raddef}.
\end{lemma}

There is a hyperbolic extension  of this inequality that we will need.  
For $ X = (x_1 ,\dotsc, x_n) \in U ^{n}$ and $ I = (i_1 ,\dotsc, i_n) \in \mathbb N ^{n}$, set 
\begin{equation*}
r _{I} (X) = \prod _{t=1} ^{n} r _{i_t} (x_t) \,. 
\end{equation*}

\begin{lemma}\label{l:hyperbolic} For coefficients $ c _{I} \in \mathbb R  $, for any $ k\in \mathbb N $, 
and $ K\in \mathbb N ^{n}$  
\begin{equation*}
\Bigl\lVert \sum_{I  \in \mathbb N ^{n} \;:\; \lvert  I\rvert = k } c_I r_{I+K} (X)\Bigr\rVert_{L _q (U ^{n})} 
\le [C \sqrt q] ^{n-1} 
\Biggl[ \sum_{I \;:\; \lvert  I\rvert = k } \lvert  c_I\rvert ^2    \Biggr] ^{\frac{1}{2}}  \,. 
\end{equation*}
\end{lemma}

\begin{proof}
The point of the estimate is that we need only apply the Littlewood--Paley $ n-1$ times. 
That we can do so recursively, follows from the Hilbert space structure associated with square functions.

Indeed, apply the Littlewood--Paley inequality in the first coordinate only. We have 
\begin{align*} 
\int _{U} \Bigl\lvert  \sum_{I \;:\; \lvert  I\rvert = k } c_I r_{I+K} (X) \Bigr\rvert^q \; d x_1 
& 
\le [C \sqrt q ] ^{q} 
\Biggl[
\sum_{t \in \mathbb N }\,\,  \Bigl\lvert 
\sum_{\substack{I' \;:\; \lvert  I'\rvert = k -t }}  
 c _{(t,I')} r_{ I'+ K'} (X')
\Bigr\rvert ^2 
\Biggr] ^{\frac{q}{2}}  
\end{align*} 
On the right, we set $ K'= (k_2 ,\dotsc, k_n)$, and similarly for $ I'$ and $ X'$. 
Note that the length of $ I'$ is prescribed to be $ k-t$, 
and as well, that the sum on the right is a Hilbert space ($\ell^2$) norm of a Hilbert space-valued Rademacher series in $n-1$ variables.  
In particular, the Littlewood-Paley inequalities apply to the sum on the right etc.  
Moreover, since the length of the vectors $ I$ is fixed, in $ n-1$ applications the process terminates. 

Such arguments are common in \emph{product harmonic analysis} and have been used in the context of discrepancy in  e.g. \cite{MR2414745}, see also \cite{MR2817765}.
\end{proof}

\end{document}